\documentclass[12pt]{amsart}
\usepackage{amsfonts,amsmath,amsthm,amssymb}
\usepackage{latexsym}
\usepackage{enumerate}
\usepackage{graphics}
\usepackage{paralist}
\newtheorem{theorem}{Theorem}[section]
\newtheorem{corollary}[theorem]{Corollary}

\newtheorem{lemma}[theorem]{Lemma}
\newtheorem{example}[theorem]{Example}

\newtheorem{proposition}[theorem]{Proposition}

\newtheorem{definition}{Definition}[section]
\theoremstyle{definition}

\theoremstyle{remark}

\theoremstyle{remark}

\newcommand{\beql}[1]{\begin{equation}\label{#1}}
\newcommand{\eeq}{\end{equation}}

\begin{document}

\title{Hadamard matrices modulo 5}

\author{Moon Ho Lee and Ferenc Sz\"oll\H{o}si}

\date{July 7, 2013}

\address{M. H. Lee: Division of Electronics and Information Engineering,
Chonbuk National University, Jeonju, Republic of Korea}\email{moonho@jbnu.ac.kr}

\address{F. Sz\"oll\H{o}si: Institute of Mathematics, Department of Analysis, Budapest University of Technology and Economics, H-1111, Egry J. u. 1, Budapest, Hungary}\email{szoferi@gmail.com}

\thanks{This work was supported by the World Class University R32-2012-000-20014-0 NRF, BSRP 2010-0020942 NRF and MEST 2012-002521 NRF, Korea; and by the Hungarian National Research Fund OTKA K-77748.}


\begin{abstract}
In this paper we introduce modular symmetric designs and use them to study the existence of Hadamard matrices modulo $5$. We prove that there exist $5$-modular Hadamard matrices of order $n$ if and only if $n\not\equiv 3,7\ (\mathrm{mod}\ 10)$ or $n\neq 6, 11$. In particular, this solves the $5$-modular version of the Hadamard conjecture.
\end{abstract}

\maketitle

	
\section{Introduction}
Hadamard matrices, real or generalized, have many applications in mathematics \cite{KJ}, \cite{KM}. A real Hadamard matrix of order $n$ is an $n\times n$ matrix $H$ with $\pm1$ entries such that $HH^T=nI$ where $T$ denotes the transpose and $I$ is the identity matrix. Note that the rows and columns of Hadamard matrices are orthogonal. Here we are concerned with modular Hadamard matrices. Given a modulus $m\geq2$, an $m$-modular Hadamard matrix $H$ of size $n$ is an $n\times n$ matrix with $\pm 1$ entries such that $HH^T\equiv nI\ (\mathrm{mod}\ m)$. A modular Hadamard matrix is normalized, if the elements in its first row and column are all $1$. Modular Hadamard matrices were introduced in $1972$ by Marrero and Butson \cite{MB1}, who related these objects to various combinatorial designs and also gave several general constructions obtaining modular Hadamard matrices. Subsequently further results were achieved in \cite{RP}, \cite{BDS} and \cite{M3}. Recently, Eliahou and Kervaire in \cite{EK} proved the existence of $32$-modular Hadamard matrices for every order $n$ divisible by $4$ by using modular Golay sequences \cite{GJ}. Their efforts were motivated by two long-standing conjectures of combinatorics, namely the Hadamard conjecture, and Ryser's conjecture. The Hadamard conjecture presumes the existence of real Hadamard matrices in every doubly even order, while Ryser conjectured that there does not exist any circulant real Hadamard matrix of order $n>4$ \cite{RY}. Currently both of these conjectures are far out of reach, despite recent efforts \cite{LS}, \cite{M}. We remark here that various other combinatorial problems have their modular analogue as well: notable examples are Golomb rulers \cite{GR} (or finite Sidon sets \cite{RU}). Results from the modular setting frequently can be translated to the non-modular setting.

The concept of modular Hadamard matrices resurfaced in the engineering literature recently during the course of the investigation of jacket matrices \cite{CWH}. In particular, in reference \cite{LY} some connections to cryptographic applications were mentioned \cite{R}, \cite{S}.

The outline of this paper is as follows. After this introductory section, in Section $2$ we recall some results from the literature to briefly discuss the existence of modular Hadamard matrices of small moduli. In Section $3$ we generalize a concept of Marrero \cite{BDS}, and introduce, what we call $m$-modular symmetric designs. Additionally, we present a fairly general direct sum type construction of modular Hadamard matrices. As an application, we use this newly developed theory to decide the existence of $5$-modular Hadamard matrices.

Throughout this paper we use the shorthand notation $\mathrm{MH}(n,m)$ for denoting an $m$-modular Hadamard matrix of size $n$. By convention, a real Hadamard matrix of order $n$ is denoted by $\mathrm{MH}(n,0)$. We also use the notation $(a,b)$ to refer to the greatest common divisor of the integer numbers $a,b\geq 0$.
\section{Preliminaries}
The aim of this section is to provide the reader with an overview of the basic results on modular Hadamard matrices. It turns out that it is possible to completely decide the existence of $\mathrm{MH}(n,m)$ matrices for $m=2,3,4$ and $6$ by utilizing relatively simple methods, most of which were introduced in \cite{MB1} and \cite{MB2}. We begin with recalling some necessary conditions as follows. We denote by $\varphi(n)$ Euler's totient function, as usual.
\begin{lemma}[cf.\ \mbox{\cite[Corollary $2.1$]{MB1}}]\label{L21}
Let $H$ be a $\mathrm{MH}(n,m)$ matrix with $n\geq 3$. Then,
\begin{enumerate}[$($a$)$]
\item if $m$ is even, then $n$ is even. Moreover, if $m\equiv0\ (\mathrm{mod}\ 4)$ then $n\equiv 0\ (\mathrm{mod}\ 4)$;
\item if $m$ is odd, $n\not\equiv0\ (\mathrm{mod}\ m)$ then $n\geq 4r$, where $1\leq r\leq m-1$, such that $r\equiv 2^{\varphi(m)-2}n\ (\mathrm{mod}\ m)$.
\end{enumerate}
\end{lemma}
\begin{proof}
We can assume that $H$ is normalized. Let us denote by $A, B, C$ and $D$ the number of vertical pairs $\left[1,1\right]^T$, $\left[1,-1\right]^T$, $\left[-1,1\right]^T$, $\left[-1,-1\right]^T$ in the second and third row of $H$. Note that $A\geq1$ due to normalization. Clearly, $A+B+C+D=n$. Moreover, by considering the orthogonality conditions within the first three rows, we find that
\[A+B-C-D\equiv0\ (\mathrm{mod}\ m),\qquad A-B+C-D\equiv0\ (\mathrm{mod}\ m),\qquad A-B-C+D\equiv0\ (\mathrm{mod}\ m),\]
and consequently $4A\equiv n\ (\mathrm{mod}\ m)$. It follows that $(4,m)|n$.

On the other hand, if $m$ is odd, then we find easily (as $n\not\equiv 0\ (\mathrm{mod}\ m)$), that
\[A\equiv B\equiv C\equiv D\equiv 2^{\varphi(m)-2}n\not\equiv0\ (\mathrm{mod}\ m).\]
Therefore $n=A+B+C+D\geq 4r$, as claimed.
\end{proof}
Another useful restriction is described in the following result.
\begin{lemma}[cf.\ \mbox{\cite[Theorem $2.2$]{MB1}}]\label{L22}
Let $H$ be a $\mathrm{MH}(n,m)$ matrix. If $(n,m)=1$, $n$ is odd then $n$ is a quadratic residue of $m$.
\end{lemma}
\begin{proof}
We have $HH^T\equiv nI\ (\mathrm{mod}\ m)$ and consequently $\left(\mathrm{det}H\right)^2\equiv n^n\ (\mathrm{mod}\ m)$.
\end{proof}
Now we recall some constructions of modular Hadamard matrices. We denote by $J$ the matrix with all entries $1$, as usual.
\begin{lemma}[\mbox{\cite[Theorem $2.3$]{MB2}}]\label{L23}
If $n\equiv 0\ (\mathrm{mod}\ m)$ or $n\equiv 4\ (\mathrm{mod}\ m)$ then there exist $\mathrm{MH}(n,m)$ matrices.
\end{lemma}
\begin{proof}
The matrices $J$ and $J-2I$ are $\mathrm{MH}(n,m)$ matrices when $n$ is a multiple of $m$ or $n-4$ is a multiple of $m$, respectively.
\end{proof}
We can use the Kronecker product to obtain new matrices from old. Although throughout this paper one of the factors is always the $2\times 2$ real Hadamard matrix
\[F_2=\left[\begin{array}{rr}
1 & 1\\
1 & -1
\end{array}\right]\]
and hence we double the size of the matrices (cf.\ \cite[p.\ 87]{EK}), we state a more general result as follows.
\begin{lemma}[\mbox{\cite[Theorem $2.1$]{MB2}}]\label{L24}
Let $H$ be a $\mathrm{MH}(n_1,m_1)$ and $K$ a $\mathrm{MH}(n_2,m_2)$. Then $H\otimes K$ is a $\mathrm{MH}(n_1n_2,(m_1m_2,m_1n_2,m_2n_1))$.
\end{lemma}
The proof is taken from \cite{EK}.
\begin{proof}
Observe that $HH^T=n_1I_{n_1}+m_1X$ and $KK^T=n_2I_{n_2}+m_2Y$ for some integer matrices $X$ and $Y$. Therefore
\[(H\otimes K)(H\otimes K)^T=n_1n_2I_{n_1n_2}+m_1n_2X\otimes I_{n_2}+m_2n_1I_{n_1}\otimes Y+m_1m_2X\otimes Y.\qedhere\]
\end{proof}
With the aid of these results it is easy to decide the existence of $\mathrm{MH}(n,m)$ matrices for $m=2,3,4$ and $6$. 
\begin{theorem}[\cite{MB1}, \cite{MB2}]
Let $n\geq 2$. Then
\begin{enumerate}[$(a)$]
\item a $\mathrm{MH}(n,2)$ exist if and only if $n$ is even;
\item a $\mathrm{MH}(n,3)$ exist if and only if $n\not\equiv 5\ (\mathrm{mod}\ 6)$;
\item a $\mathrm{MH}(n,4)$ exist if and only if $n=2$ or $n$ is doubly even;
\item a $\mathrm{MH}(n,6)$ exist if and only if $n$ is even.
\end{enumerate}
\end{theorem}
\begin{proof}
On the one hand, the necessary conditions described here follow from Lemma \ref{L21} and Lemma \ref{L22} in the cases $(a), (c), (d)$ and $(b)$, respectively.

On the other hand, the existence of these modular Hadamard matrices follows from Lemma \ref{L23} almost immediately, except for the case $n\equiv\ 2\ (\mathrm{mod}\ 6)$ in parts $(b)$ and $(d)$. We construct such matrices via Lemma \ref{L24} by taking the Kronecker product of the $2\times 2$ Hadamard matrix $F_2$ with $\mathrm{MH}(3k+1,3)$ matrices.
\end{proof}
We remark that combination of the above ideas lead to the determination of $\mathrm{MH}(n,12)$ matrices as well \cite{EK}.
\section{Modular symmetric designs and Hadamard matrices modulo $5$}
In this section we introduce modular symmetric designs and use them to investigate the existence of $5$-modular Hadamard matrices. The next result easily follows from the theory we reviewed in Section $2$.
\begin{corollary}\label{trivC}
There exist $\mathrm{MH}(n,5)$ matrices if $n\equiv 0,4,5,8,9\ (\mathrm{mod}\ 10)$. There do not exist $\mathrm{MH}(n,5)$ matrices if $n\equiv 3,7\ (\mathrm{mod}\ 10)$.
\end{corollary}
\begin{proof}
The existence of orders $n\equiv 0,4,5,9\ (\mathrm{mod}\ 10)$ follow from Lemma \ref{L23}. Matrices of order $n\equiv 8\ (\mathrm{mod}\ 10)$ can be obtained by considering the Kronecker product of a $\mathrm{MH}(5k+4,5)$ matrix with the $2\times 2$ Hadamard matrix $F_2$ via Lemma \ref{L24}. On the other hand, the cases $n\equiv 3,7\ (\mathrm{mod}\ 10)$ are eliminated by Lemma \ref{L22}.
\end{proof}
It appears that addressing the remaining cases is a nontrivial problem. In particular, we have some further nonexistence results in the cases $n\equiv 1,6\ (\mathrm{mod}\ 10)$ due to Lemma \ref{L21}.
\begin{corollary}[cf.\ \mbox{\cite[Theorem $3.1$]{MB2}}]\label{NEX}
There does not exist $\mathrm{MH}(6,5)$ and $\mathrm{MH}(11,5)$ matrices.
\end{corollary}
\begin{proof}
Suppose, to the contrary, that a $\mathrm{MH}(n,m)$ exists with $m=5$ and $n=6,11$. Then, by an application of Lemma \ref{L21} with noting that $\varphi(5)=4$ and $n\equiv 1\ (\mathrm{mod}\ m)$ in both cases we arrive to the (same) lower bound $n\geq 16$, a contradiction.
\end{proof}
Therefore the real challenge is to construct $\mathrm{MH}(n,5)$ matrices when $n\equiv 1,2$ or $6\ (\mathrm{mod}\ 10)$. To obtain interesting examples of modular Hadamard matrices it is natural to consider combinatorial designs \cite{RP}, \cite{M3}, \cite{MB1} and \cite{MB2}. It turns out, however, that the relevant mathematical object is the following relaxed concept of modular symmetric designs.
\begin{definition}[cf.\ \cite{BDS}]
Let $m,v\geq2$ be an integer. A $v\times v$ matrix $D$ with entries $0$ or $1$ is called an $m$-modular symmetric design, if there exist integer numbers $k$ and $\lambda$, such that $DD^T\equiv (k-\lambda)I+\lambda J\ (\mathrm{mod}\ m)$ and $DJ\equiv JD\equiv kJ\ (\mathrm{mod}\ m)$. We denote these objects by $(v,k,\lambda;m)$, and refer to them as $(v,k,\lambda;m)$ designs.
\end{definition}
In other words, $D$ is an $m$-modular symmetric design, if the number of $1$s in each row and column is congruent to $k\ (\mathrm{mod}\ m)$, and the number of vertical pairs $\left[1,1\right]^T$ within two different rows is congruent to $\lambda\ (\mathrm{mod}\ m)$. Clearly, any symmetric $(v,k,\lambda)$ design is a $(v,k,\lambda;m)$ design for all $m\geq2$. Other examples can be obtained from modular difference sets \cite{BDS}. Note, however, that modular symmetric designs constructed from modular difference sets have the same number of $1$s in each row and column, thus constitute a special case of our concept. The reader is advised to consult \cite[Chapter II.6]{CRC}, where the general theory of symmetric designs is presented, along with detailed summarizing tables of the parameters of the known symmetric designs of small orders. We give here a non-trivial example as follows.
\begin{example}\label{EX1}
Consider the $(13,4,1)$ design $R$, generated by the cyclic permutations of the row vector  $[1,0,1,1,0,0,0,1,0,0,0,0,0]$. Then, the following is a $(26,1,2;5)$$:$
\[\left[\begin{array}{cc}
R & J-I\\
J-I & J-R^T
\end{array}\right].\]
\end{example}
Let $n\geq2$, and consider a normalized $\mathrm{MH}(n,m)$ matrix $H$. By discarding its first row and column we obtain the core of $H$. Under some mild assumptions we can obtain an $m$-modular design from the core of $H$ with unique parameters, as follows.
\begin{lemma}\label{corel}
Let $m,n\geq 3$, $(m,n)=1$. If $H$ is a normalized $\mathrm{MH}(n,m)$, whose core is denoted by $C$, then $D=(C+J)/2$ is a $(n-1,2^{\varphi(m)-1}(n-2),2^{\varphi(m)-2}(n-4);m)$ design.
\end{lemma}
\begin{proof}
First observe that $m$ is necessarily odd by Lemma \ref{L21}. As $(m,n)=1$ we have $HH^T\equiv H^TH\equiv nI\ (\mathrm{mod}\ m)$. In particular, the columns of $H$ are pairwise orthogonal modulo $m$. It follows that $CJ\equiv JC\equiv -J\ (\mathrm{mod}\ m)$, and consequently
\[2DJ\equiv 2JD\equiv JC+J^2\equiv(n-2)J\ (\mathrm{mod}\ m).\]
Secondly, we have $CC^T\equiv nI-J\ (\mathrm{mod}\ m)$, and hence
\[4DD^T=(C+J)(C+J)^T=CC^T+JC^T+CJ+(n-1)J\equiv nI+(n-4)J\ (\mathrm{mod}\ m).\]
The statement follows after multiplying these equations by $2^{\varphi(m)-1}$ and $2^{\varphi(m)-2}$, respectively.
\end{proof}
Combinatorial designs are extremely useful for our purposes. We state here a simple result as follows.
\begin{lemma}[cf.\ \mbox{\cite[Theorem 2.5]{MB2}}]\label{const1}
Let $D$ be a $(v,k,\lambda;m)$ design. Then the matrix $2D-J$ is a $\mathrm{MH}(v,m)$ if and only if $v\equiv 4(k-\lambda)\ (\mathrm{mod}\ m)$.
\end{lemma}
\begin{proof}
\[(2D-J)(2D-J)^T\equiv 4(k-\lambda)I+(v-4k+4\lambda)J\equiv vI\ (\mathrm{mod}\ m).\qedhere\]
\end{proof}
However, it is difficult to obtain combinatorial designs, and thus the applications of Lemma \ref{const1} are somewhat limited. To get more powerful construction methods, we combine two $m$-modular symmetric designs as follows.
\begin{definition}
Let $D_1$ and $D_2$ be a $(v_1,k_1,\lambda_1;m)$ and a $(v_2,k_2,\lambda_2;m)$ design, respectively. Then their direct sum, denoted by $D_1\oplus D_2$, is the following block matrix of size $v_1+v_2$:
\[\left[\begin{array}{cc}
D_1 & J\\
J^T & D_2
\end{array}\right].\]
\end{definition}
Note that the direct sum of modular designs is not a modular design in general. Nevertheless, it is worthwhile to characterize the cases when the direct sum of two modular designs leads to a modular Hadamard matrix.
\begin{lemma}\label{lemlemma}
Let $v_1,v_2\geq 2$, $D_1$ and $D_2$ be a $(v_1,k_1,\lambda_1;m)$ and a $(v_2,k_2,\lambda_2;m)$ design, respectively. Then $2(D_1\oplus D_2)-J$ is a $\mathrm{MH}(v_1+v_2,m)$ if and only if
\begin{align*}
v_2&\equiv-v_1+4k_1-4\lambda_1\ (\mathrm{mod}\ m),\\
2k_2&\equiv 2k_1-4\lambda_1\ (\mathrm{mod}\ m),\\
4\lambda_2&\equiv-4\lambda_1\ (\mathrm{mod}\ m).
\end{align*}
\end{lemma}
\begin{proof}
The formulas follow directly from the orthogonality conditions of the rows of $2(D_1\oplus D_2)-J$ and from the fact that for $v_1,v_2\geq 2$ and $m\geq2$ the matrices $J$ and $I$ are linearly independent. In particular, we find that
\begin{gather*}
v_1+v_2-4k_1+4\lambda_1\equiv0\ (\mathrm{mod}\ m),\\
v_1+v_2-4k_2+4\lambda_2\equiv0\ (\mathrm{mod}\ m),\\
v_1+v_2-2k_1-2k_2\equiv0\ (\mathrm{mod}\ m),
\end{gather*}
must hold. The desired result follows after some easy manipulation.
\end{proof}
We provide the reader with an illustrative example as follows.
\begin{example}\label{HEL}
Here we construct a $\mathrm{MH}(86,5)$ matrix. Let $D_1$, $D_2$ and $D_3$ be a $(16,6,2)$, a $(35,17,8)$ and a $(36,21,12)$ design, respectively $($see \cite[p.\ 740]{CRC} and \cite[pp.\ 273--274]{CRC}$)$. First we consider $2(D_1\oplus D_2)-J$, which is a $\mathrm{MH}(51,5)$ by Lemma \ref{lemlemma}. The core of this matrix, after normalization, leads to a $(50,2,3;5)$ design $D_4$ by Lemma \ref{corel}. The desired $\mathrm{MH}(86,5)$ matrix can be obtained by considering $2(D_3\oplus D_4)-J$.
\end{example}
Now we give a construction of $\mathrm{MH}(n,5)$ matrices.
\begin{proposition}\label{P39}
There exist $\mathrm{MH}(n,5)$ matrices of order $n\equiv 1\ (\mathrm{mod}\ 5)$ if and only if $n\neq6,11$.
\end{proposition}
\begin{proof}
Observe that if $n=20k+16=4(5k+4)$ then we can get $\mathrm{MH}(20k+16,5)$ matrices for every $k\geq 0$ by doubling twice the $\mathrm{MH}(5k+4,5)$ matrices $J-2I$ via Lemma \ref{L24}. Now we use these matrices, or more precisely, the corresponding $5$-modular designs (arising from Lemma \ref{corel}) with parameters $(20k+15,2,3;5)$ as follows. We take their direct sum with the designs $(26,1,2;5)$, $(91,81,72)$ and $(16,6,2)$ to obtain, after a reference to Lemma \ref{lemlemma}, $\mathrm{MH}(20k+41,5)$, $\mathrm{MH}(20k+106,5)$ and $\mathrm{MH}(20k+31,5)$, respectively. The first modular design is provided in Example \ref{EX1}, the second one is the complement of a projective plane of size $91$, while the third one is a Menon design of size $16$ \cite{CRC}.

Finally, we take care of the small order cases as follows: a $\mathrm{MH}(1,5)$ is just the $1\times 1$ matrix $F_1=\left[\begin{array}{c}1\end{array}\right]$; $\mathrm{MH}(6,5)$ and $\mathrm{MH}(11,5)$ do not exist by Corollary \ref{NEX}. A $\mathrm{MH}(21,5)$ and $\mathrm{MH}(26,5)$ can be constructed from a $(21,5,1)$ and from the $(26,1,2;5)$ design (given in Example \ref{EX1}) via Lemma \ref{const1}, respectively. A $\mathrm{MH}(46,5)$ and a $\mathrm{MH}(66,5)$ can be obtained via Lemma \ref{lemlemma} as follows. Let $D_1, D_2, D_3$ and $D_4$ be a $(26,1,2;5)$, $(20,2,3;5)$, $(21,5,1)$ and a $(45,33,24)$ $5$-modular design, respectively. Then, $2(D_1\oplus D_2)-J$ and $2(D_3\oplus D_4)-J$ are the desired $\mathrm{MH}(46,5)$ and $\mathrm{MH}(66,5)$ matrices, respectively. Here the $5$-modular design $(20,2,3;5)$ can be obtained from the core of a $\mathrm{MH}(21,5)$ matrix by Lemma \ref{corel}, while the $(45,33,24)$ design is the complement of the one listed in \cite[p.\ 119]{CRC}. Finally, a $\mathrm{MH}(86,5)$ matrix was constructed in Example \ref{HEL}.
\end{proof}
From Proposition \ref{P39} the existence of $\mathrm{MH}(5k+2,5)$ matrices follows immediately when $k$ is even (and the case $k$ odd is impossible due to Lemma \ref{L22}).
\begin{corollary}\label{C310}
There exist $\mathrm{MH}(n,5)$ for every $n\equiv 2\ (\mathrm{mod}\ 10)$.
\end{corollary}
\begin{proof}
The $\mathrm{MH}(20k+2,5)$ and $\mathrm{MH}(40k+12,5)$ matrices can be obtained by doubling the $\mathrm{MH}(10k+1,5)$ and $\mathrm{MH}(20k+6,5)$ matrices of Proposition \ref{P39} via Lemma \ref{L24}, respectively, while $\mathrm{MH}(40k+32,5)$ matrices can be obtained by doubling three times the $\mathrm{MH}(5k+4,5)$ matrices via Lemma \ref{L24}.
\end{proof}
We have finished the discussion of $5$-modular Hadamard matrices. The main result of the paper follows.
\begin{theorem}\label{T12345}
There exist $\mathrm{MH}(n,5)$ matrices if and only if $n\not\equiv 3, 7\ (\mathrm{mod}\ 10)$ or $n\neq 6, 11$.
\end{theorem}
\begin{proof}
Follows immediately from Corollary \ref{trivC}, Proposition \ref{P39} and Corollary \ref{C310}.
\end{proof}
In particular, the Hadamard conjecture modulo $5$ is true (cf.\ \cite{EK}).
\begin{corollary}
For every $k\geq 1$ there exist $5$-modular Hadamard matrices of order $4k$.
\end{corollary}
\begin{proof}
Follows immediately from Theorem \ref{T12345}.
\end{proof}
We believe that the new tools and ideas presented in this paper are powerful enough to completely decide the existence of $\mathrm{MH}(n,m)$ matrices for some further values of $m$ as well. In light of our results, however, it seems that presenting infinite constructions and dealing with several exceptional cases of relatively small order might be equally difficult in general.

\end{document}